\documentclass[10pt,reqno]{amsart}

\usepackage{amsmath}
\usepackage{amsfonts}
\usepackage{amssymb}
\usepackage{amsthm}
\usepackage{mathtools}

\usepackage{stmaryrd}

\usepackage{mathrsfs}
\usepackage{latexsym}
\usepackage{verbatim} 
\numberwithin{equation}{section}
\usepackage{enumitem}

\usepackage{graphicx}
\usepackage{subfigure}
\usepackage{epstopdf}

\usepackage{empheq}

\usepackage{ifthen} 

\provideboolean{shownotes} 
\setboolean{shownotes}{true} 
\newcommand{\margnote}[1]{
\ifthenelse{\boolean{shownotes}}%
{\marginpar{\raggedright\tiny\texttt{#1}}}%
{}%
}
\newcommand{\hole}[1]{
\ifthenelse{\boolean{shownotes}}%
{\begin{center} \fbox{ \rule {.25cm}{0cm}
\rule[-.1cm]{0cm}{.4cm} \parbox{.85\textwidth}{\begin{center}
\texttt{#1}\end{center}} \rule {.25cm}{0cm}}\end{center}}
{}
}

\theoremstyle{plain}

\newtheorem{lemma}{Lemma}[section]
\newtheorem{theorem}[lemma]{Theorem}
\newtheorem{proposition}[lemma]{Proposition}
\newtheorem{corollary}[lemma]{Corollary}

\theoremstyle{definition}

\newtheorem{remark}[lemma]{Remark}
\newtheorem{definition}[lemma]{Definition}

\theoremstyle{remark}

\usepackage{cite} 

\usepackage[colorlinks=true,urlcolor=blue,
citecolor=red,linkcolor=blue,linktocpage,pdfpagelabels,
bookmarksnumbered,bookmarksopen]{hyperref}
\usepackage{cleveref}
\usepackage{orcidlink}

\usepackage[pagewise,mathlines]{lineno}

\newcommand{\R}{\mathbb{R}}
\newcommand{\C}{\mathbb{C}}

\newcommand{\N}{\mathbb{N}}

\newcommand{\cL}{{\mathcal{L}}}
\newcommand{\cA}{{\mathcal{A}}}
\newcommand{\caB}{{\mathcal{B}}}
\newcommand{\cM}{{\mathcal{M}}}

\newcommand{\sgn}{\mathrm{sgn}\,}

\usepackage{scalerel}

\newcommand{\ep}{\epsilon}

\newcommand{\ShowColoredChanges}{true} 
\ifthenelse{\isundefined{\ShowColoredChanges}}
  {

  }
  {

  }

\begin{document}

\title[Analysis of the sine-Gordon equation with a $\delta$-potential]{Analysis of the sine-Gordon equation with a nonlinear $\delta$-potential}

\author[S. Moroni]{Sergio Moroni \orcidlink{0009-0005-1586-632X}}

\address{{\rm (S. Moroni)} Basque Center for Applied Mathematics\\Alameda de Mazarredo 14, 48009\\ Bilbao, Bizkaia (Spain)}

\email{smoroni@bcamath.org}

\author[R. G. Plaza]{Ram\'on G. Plaza \orcidlink{0000-0001-8293-0006}}

\address{{\rm (R. G. Plaza)} Departamento de Matem\'aticas y Mec\'anica\\Instituto de Investigaciones en Matem\'aticas Aplicadas y en Sistemas\\Universidad Nacional Aut\'onoma de M\'exico\\Circuito Escolar s/n, Ciudad Universitaria\\C.P. 04510 Cd. de M\'{e}xico (Mexico)}

\email{plaza@aries.iimas.unam.mx}

\begin{abstract}
This paper is devoted to the analysis of the following nonlinear wave equation
\[
u_{tt} - u_{xx} + (1 + q\delta_0(x)) \sin u = 0,
\]
where $\delta_0 = \delta_0(x)$ is the Dirac delta function centered at the origin and $q \in \R$ is a constant. Equations of this form arise in the study of propagating solitons in the presence of a localized inhomogeneity. It is proved that the Cauchy problem for this equation is globally well-posed in the energy space $H^1_{\sin} \times L^2$. A complete characterization of stationary waves in the energy space, based on the parameter $q$, is also provided. Finally, a criterion to determine the stability or instability of the stationary waves, which depends upon the sign of the parameter $q$, is established.
\end{abstract}

\keywords{sine-Gordon equation; delta nonlinearity; well-posedness; stationary waves; stability}

\subjclass[2020]{35Q51; 35J61; 47E05}

\maketitle
\setcounter{tocdepth}{1}



\section{Introduction}
\label{secintro}

In this paper we study the evolution nonlinear equation
\begin{equation}
\label{sGdelta}
u_{tt} - u_{xx} + (1 + q\delta_0(x)) \sin u = 0, \qquad x \in \R, \quad t > 0,
\end{equation}
for an unknown $u = u(x,t)$, where $q \in \R$ is a constant, $q \neq 0$, and $\delta_0 = \delta_0(x)$ stands for the Dirac delta function centered at $x = 0$. In the case when $q=0$, equation \eqref{sGdelta} reduces to the classical scalar field equation,
\begin{equation}
\label{sG}
u_{tt} - u_{xx} + \sin u  = 0, 
\end{equation}
known as the \emph{sine-Gordon model} in laboratory coordinates. 

The sine-Gordon equation \eqref{sG} is one of the fundamental models in mathematical physics with a broad range of applications, varying from negative Gaussian curvature manifold theory \cite{Eis1909} to the dynamics of the magnetic flux in a Josephson line \cite{BEMS71,SCR76}, among others. It is a nonlinear wave equation underlying many important mathematical features, such as complete integrability \cite{AKNS73,AKNS74}, a Hamiltonian structure \cite{TaFa76} and the existence of localized solutions (solitons) \cite{SCM73,Sco03}. We refer the reader to the monographs by Lamb \cite{Lamb80} and by Dauxois and Peyrard \cite{DaPe10} for further readings on the topic.
%
%
%
%
%

Being equation \eqref{sG} completely integrable, it admits a Lax pair formulation (cf. \cite{AKNS73,AKNS74}).  Moreover, it supports a variegated family of explicit soliton solutions, such as the time-independent \emph{topological kink} \cite{SCM73,Sco03},
\begin{equation} 
K(x) = 4\arctan e^x,
\label{Kink}
\end{equation}
as well as more intrincated configurations such as breathers and wobbling kinks (see, e.g., \cite{Sco79,Enz85,CQS11} and the references therein). Invariance with respect to Lorentz transformations leads to the existence, by application of Lorentz boost to $K$, of a family of traveling waves or moving kinks,
\begin{equation}
\label{kinky}
K_v (x,t) = 4 \arctan \Big( \exp \Big( \frac{x-vt - x_0}{\sqrt{1-v^2}}\Big) \Big),
\end{equation}
traveling with any subluminal speed $v\in (-1,1)$. 

The long time dynamics of solutions to the sine-Gordon equation has stimulated an intense investigation in the last 50 years and it is not our intention to make an exhaustive literature review here. However, it is worth mentioning that the orbital stability of the kink, by now a classical statement, was originally proved by Henry \emph{et al.} \cite{HPW82}, and that a more recent and different proof was presented in \cite{KMMVdB21}. For other explicit configurations the argument is more delicate, and it relies on the integrability of \eqref{sG}: either through the infinitely many conserved quantities for breathers \cite{AMuP17a}, or through the use of B\"acklund transform for multikink \cite{MuPa19} and wobbling kinks \cite{AMuP23}. The question of asymptotic stability is significantly harder. Recently in \cite{AMuP23} the authors have proved local asymptotic stability when the initial datum lies in a parity manifold; full asymptotic stability was proved, upon application of different strategies, by \cite{LuSch23,ChLuh24,KoYu23,ChLiuLu22} for initial data in weighted Sobolev spaces.

\subsection{Disordered systems in heterogenous media}

The model under consideration in this paper (equation \eqref{sGdelta}) is motivated by the theory of nonlinear excitations in disordered systems such as soliton interactions in media with localized inhomogeneities (or impurities); cf. Kivshar \cite{Kiv92}. For instance, in order to describe the presence of a point-like inhomogeneity along a Josephson line, the following model equation has been considered (cf. Fei \emph{et al.} \cite{FKV92}),
\begin{equation}
\label{sgdeltaep}
u_{tt} - u_{xx} + \sin u = \varepsilon \delta_0(x) \sin u,
\end{equation}
where $u$ is the normalized magnetic flux and the $\delta$-forcing models the action of an attractive (for $\varepsilon > 0$) defect localized at the origin. In the physics literature (see, e.g., \cite{FKV92,FKV93,KSV93,ZKMV91}), it is often assumed that $\varepsilon$ is positive and sufficiently small in order to account for the attractive nature of a small material defect. Clearly, equation \eqref{sgdeltaep} is a particular case of \eqref{sGdelta} with $q = - \varepsilon$.

Up to our knowledge, the most general and realistic equation with nonlinear $\delta$-localized potentials that models local regions of high-speed Josephson junctions (such as ``microshorts'' or thin spots of the electron barrier) was first proposed by McLaughlin and Scott \cite{McLSc78a}, for which equation \eqref{sgdeltaep} represents the particular case of a single impurity. Equation \eqref{sgdeltaep} was subsequently studied by Fei \emph{et al.} \cite{FKV92,KFV91} (see also \cite{ZKMV91}) to account for the interaction of a propagating wave with a strongly localized singularity. In particular, the authors in \cite{FKV92} study a kink moving towards the localized attractive impurity.  The nonlinearity $\sin u$ assures that the impurity significantly interacts with the kink only when the central corp of the traveling wave is close to the singular point. The paper presents a numerical approximation of the evolution, which suggests that the moving kink can trespass the singularity, or be reflected or captured by it. The different possible outcomes depend on the initial velocity of the kink. Later, Goodman \emph{et al.} \cite{GHW02,GoHa04} studied the same problem under a dynamical systems approach, proving that an approximation of \eqref{sGdelta} predicts dynamics in good coherence with the numerical results. More recently, Gomide \emph{et al.} \cite{GoGS20} improved a formal argument about the integral approximation, giving major robustness to the dynamical result. The cited works start with an \emph{anstaz} that simplifies a solution of \eqref{sGdelta} as the sum, for all times, of a moving kink of the form \eqref{kinky} and an oscillating internal mode around the singularity, also known as the \emph{impurity mode} (cf. \cite{KFV91}) and which is essentially a breather-type approximated solution, localized in space and periodic in time. From the scalar field equation \eqref{sGdelta} they deduce an effective Lagrangian which neglects small order terms for $|\varepsilon| \ll 1$;  the dynamics is then simplified to a coupled system of ODEs for the evolution of the internal mode and kink position. Most of the literature is devoted to study of this kink/impurity interaction under different methodologies and perspectives, in view that inhomogeneities give rise to effective potentials that affect soliton dynamics. In spite of its importance, however, we have not found an analytical study of the basic properties of equations of the form \eqref{sGdelta}.

The huge interest in evolution equations with singular interactions has motivated a robust mathematical investigation on the topic. Within a wider framework, dispersive PDEs with $\delta$ singularities have been used to model several physical configurations, varying from localized impurities to short range interactions. The reader is referred to the now classical monograph by Alveberio \emph{et al.} \cite{AGHH2ed} for further information. It is to be observed that for a nonlinear $\delta$-interaction, application of self adjoint extensions is delicate, in view that the domain of the operator may depend on time.  The Schr\"odinger equation with a concentrated nonlinearity has been studied in dimensions one \cite{AdTe01}, three \cite{ADFT03,AdNOr13} and more recently in dimension two as well (cf. \cite{CCT19,ACCT21}). For a wider treatment of the argument we refer to the review paper \cite{Tenta23} and the references therein. Stability and global attractor properties of stationary waves for the Klein-Gordon equation were proved in \cite{KoKo07a,CoKo21} in dimension one, and in \cite{KopKom19} in dimension three.

\subsection{Main results}

The purpose of this work is to initiate a rigorous study for equations of the form \eqref{sGdelta}. Specifically we prove global well posedness of the Cauchy problem, we characterize stationary waves and prove their stability or instability depending on the sign of $q$. We consider these results as a classical first step in the investigation of more difficult evolution properties for equation \eqref{sGdelta}. We perform the analysis for all the parameter values $q\in \R$ and not only for attractive small impurities. In vectorial form the equation is written as a first order Cauchy problem for $u_1:=u, u_2 :=u_t$, namely,
\begin{equation}     
\label{CauchyDeltaV}
\begin{cases}
\partial_t u_1= u_2,
     \\\partial_t u_2= \partial_{xx} u_1 - (1+ q\delta_0) \sin u_1.
     \end{cases}
\end{equation}

Our first result pertains to the global well-posedeness for the Hamiltonian system \eqref{CauchyDeltaV}, with conserved energy
\begin{equation}
\label{EnerDef}
    E(u_1,u_2)(t):= \int_\R \tfrac{1}{2}|u_2|^2 + \tfrac{1}{2}|\partial_x u_1|^2  + (1- \cos u_1) \, dx + q \left ( 1-\cos u_1(0) \right).
\end{equation}
The expression for the energy suggests $H^1_{\sin} \times L^2$ as the natural energy space for the Cauchy problem, just like in the classical sine-Gordon case, where
 \[
 H^1_{\sin}(\R) := \left \{ u  \in \dot H^1\, : \, \sin\left(\tfrac{1}{2} u\right ) \in L^2(\R) \right \}.
 \]

The global existence result can be formulated as follows.
\begin{theorem}
\label{TeoExi}
For any initial datum $(u_1^0, u_2^0) \in H^1_{\sin} \times L^2$ there exists a unique global solution $(u_1,u_2) \in C([0,\infty); H^1_{\sin} \times L^2)$ to the Cauchy problem \eqref{CauchyDeltaV}, which verifies
\begin{equation}
    \label{Energyineq2}
   \sup_{t\in\R} \left (\|\sin u_1(t)\|_{L^2_x}+\|u_2(t)\|_{ L^2_x}+\|\partial_x u_1(t)\|_{L^2_x}\right )\le C \| (u_1^0, u_2^0) \|_{H^1_{\sin} \times L^2}.
\end{equation}
Moreover, this solution conserves the energy, 
\begin{equation}
\label{Ener} 
E(u_1(t), u_2(t)) = E\left ( u_1^0, u_2^0\right ), \qquad \text{for all } t.
\end{equation}
\end{theorem}

%


The proof of Theorem \ref{TeoExi} relies on an auxiliary problem with a mean field interaction approximating the $\delta$-term.  Regularity of the new problem implies conservation of energy, and hence global existence. It is to be observed that convergence of mean field approximation to the $\delta$-evolution equation has been used in specific contexts to validate the robustness and physical meaning of the model (see, e.g., \cite{CFNT14,CFNT17,ACCT21}). In the present paper, the regular approximation plays only the role of a technical tool to prove Theorem \ref{TeoExi}, following Komech and Komech \cite{KoKo07b}.

In our second result, we present a complete characterization of stationary waves of \eqref{sGdelta} in the energy space. As the system is not invariant with respect to Lorentz transformations or to the action of the $U(1)$ group, this amounts to study time independent solutions to the elliptic equation,
\begin{equation}
    \label{staticDelta}
    - u_{xx} + (1 + q \delta_0(x) ) \sin u = 0.
\end{equation}

Notice that if $u\in H^1_{\sin}$ solves \eqref{staticDelta} then $(u,0)$ is a stationary solution of the Cauchy problem \eqref{CauchyDeltaV}. In the free case $q=0$, the unique solution up to symmetries with finite energy is the kink.  The solution has a topological degree, as it connects the constant solutions with finite energy, $u=0$ and $u=2\pi$. In the perturbed case, we prove that the only static solution with the same topological degree is the centered kink. Moreover, for $|q|$ large enough, the singular condition in the origin allows for the existence of a stationary wave in $H^1$, in contrast with the classical model.

We have, consequently, the following result.

\begin{proposition}
\label{CarSW}
   If $|q|\le 2$, there exist no stationary solutions to \eqref{staticDelta} in the space $H^1(\R)$. If $|q|>2$ there exists a unique positive stationary wave in $H^1(\R)$, solution of \eqref{staticDelta}, given by
    \begin{equation}
        \label{formulaGS}
        Q(x)=  4\arctan \left( \sqrt{\frac{q+2}{q-2}} e^{-|x|}\right ).
    \end{equation}
For any $q\neq 0$ the unique $u\in H_{\sin}^1(\R)$ solution of \eqref{staticDelta} with 
\[
\lim_{x\to - \infty} u(x)= 0, \qquad \lim_{x\to \infty} u(x) = 2\pi,
\]
is the static kink, $K(x) = 4\arctan e^x$. 
\end{proposition}

The proof of Proposition \ref{CarSW} is constructive. The perturbation of the sine-Gordon equation is localized at the point $x=0$. Hence a stationary solution of \eqref{sGdelta} solves the classical sine-Gordon equation in the two semiaxes $\left \{ x>0 \right \}, \ \left \{ x<0 \right \}$, while the $\delta$ term imposes a nonlinear gluing condition in $x=0$ between the two branches (see Figure \ref{figQfunction} for a depiction of the stationary wave solution \eqref{formulaGS} for $q = -4$). 

\begin{figure}[t]
\begin{center}
\includegraphics[scale=0.6]{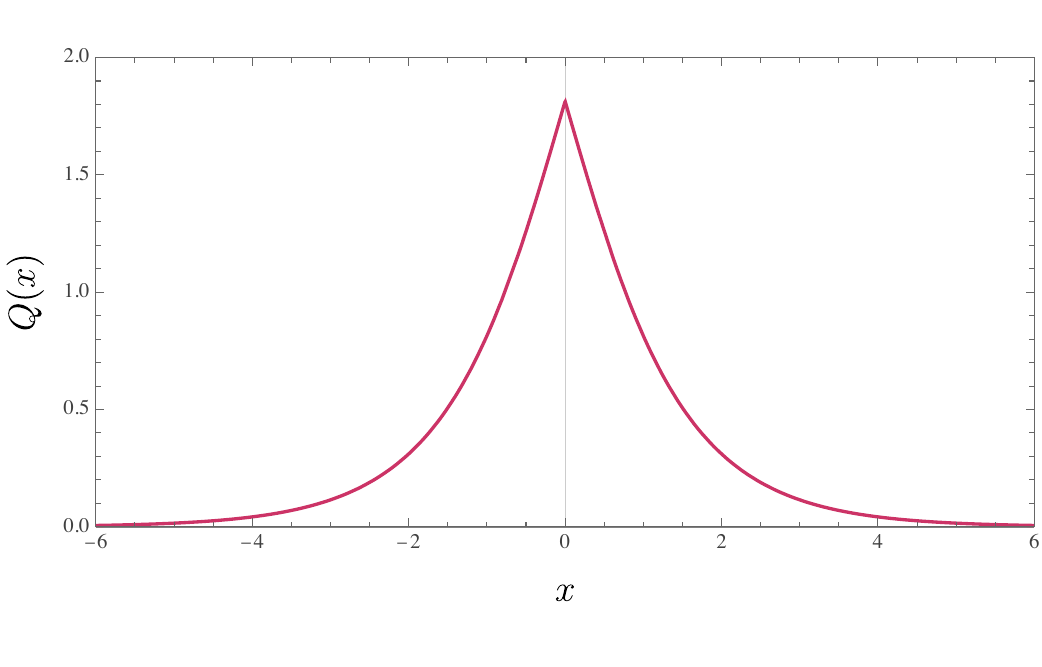}
\caption{\small{Plot of the stationary wave \eqref{formulaGS} (in red) for the parameter value $q = -4$ (color online). }
}
\label{figQfunction}
\end{center}
\end{figure}

Existence or nonexistence of a ground state for different values of $q$ correspond to the possibility, for a properly translated kink, to satisfy the gluing condition. For $q<0$, we also give a variational characterization of the stationary waves; this enforces the interest in the stationary waves for the evolution system, as they appear as minimal points of the energy. 

\begin{proposition}
Let $q<0$ and consider the minimization problems
\begin{align}
\label{MinGS}
c:&= \inf_{u\in H^1}  E(u,0), \\ 
d:&= \inf_{\substack{ u\in H_{\sin}^1  \\ u(-\infty) =0 \\ u(\infty) =2\pi}} E(u,0).
\label{MinKink}
\end{align}
Then for $-2\le q <0$ the infimum in problem \eqref{MinGS} is reached by $0$,  while for $q<-2$ it holds $0>c= E(Q,0)$. Similarly, for any $q<0$, the infimum problem is reached by the kink: $d= E(K,0)$.
\label{MinSW}
\end{proposition}

The Euler-Lagrange equation for both minimization problems, \eqref{MinGS} and \eqref{MinKink}, is the static equation \eqref{staticDelta}. Absence of a nontrivial minimizer for $-2\le q<0$ is coherent with the absence of ground states for similar values of $q$ as stated in Proposition \ref{CarSW}.

Finally, we state the third main result of the paper, classifying stability and instability for the static solutions. We give first the definition of stability in our setting.

\begin{definition}
\label{defst}
We say a stationary wave $(\Psi,0)$ is stable if there exist $\ep_0>0$, $C>0$ such that for any $\ep<\ep_0$ the condition 
\[
\left \| (u_1^0 - \Psi , u_2^0) \right \|_{H^1\times L^2} \le \ep 
\]
implies for $(u_1, u_2)$, solution to \eqref{CauchyDeltaV} with initial datum $\left (u_1^0, u_2^0 \right ) $, that
\[
\left \|(u_1 -\Psi , u_2)(t) \right \|_{H^1\times L^2} \le C \left \| (u_1^0 - \Psi , u_2^0) \right \|_{H^1\times L^2} .
\]
A stationary wave is unstable if there exists $\ep>0$ and a sequence of initial data, $\left (u_{1,n}^0,u_{2,n}^0 \right )$, $n \in \N$, converging to $(\Psi,0)$ in $H^1\times L^2 $ such that
\[
\sup_{t \in \R} \left \| \left (u_{1,n}-\Psi, u_{2,n} \right ) (t) \right \|_{H^1 \times L^2} \ge \ep .
 \] 
\end{definition}

\begin{remark}
As it is expected, the localized potential in \eqref{CauchyDeltaV} breaks the translation invariance of the system; consequently, the stability definition does not take into account any modulation parameter. 
\end{remark}

\begin{remark}
Our definition differs from the standard stability requirement in the literature, usually stated as the negation of the instability condition in Definition \ref{defst}. 
We state a stronger direct proportionality between the norms of the initial perturbation and its evolution; see the orbital stability result by Kowalczyk \emph{et al.} \cite{KMMVdB21} for a similar definition. The improved bound naturally arises from the method of proof.
\end{remark}

The criterion for stability or instability of the underlying stationary wave is stated as follows.
\begin{theorem}
\label{TeoSt}
Let $q>0$ (resp. $q>2$). Then the stationary wave $(K,0)$ (resp. $(K,0)$ and $(Q,0)$) is unstable. Conversely, let $q<0$ (resp. $q<-2$). Then the stationary wave $(K,0)$ (resp. $(K,0)$ and $(Q,0)$) is stable. 
\end{theorem}

Starting with the seminal works \cite{Wein85,ShSt85,Sht83,HPW82}, stability of stationary waves is strictly connected to the spectral properties of the operator describing the linearized equation around the wave. Loosely speaking, if the spectrum is strictly positive the wave is expected to be stable, while a negative eigenvalue is a clue for instability. It is well known that the linearized operator around a sine-Gordon kink has point spectrum $\left \{0\right \}$, which can be neglected by modulation in the stability argument, and essential spectrum $[1,\infty)$ with a resonance at the bottom. The spectral study is more delicate in the $\delta$ case. The formal expression of the operator does not take value in $L^2$,  so even proper definition of the spectrum relies on self adjoint extension theory. Identification of the right domain plays a key role in the construction of the self adjoint extension; consequently, it also determines point spectrum elements, often the crucial aspect for stability results. An additional difficulty comes from the nonlinearity in front of the $\delta$ term, which makes the domain dependent on the stationary wave itself. Once the spectral study is settled, nonlinear stabilities in the sense of Definition \ref{defst} follow directly. We were inspired by the result of Henry \emph{et al.} \cite{HPW82} for the instability theorem (see also \cite{ShSt00}) and by Martel \emph{et al.} \cite{MaMeT02} for the stability argument.

\begin{remark}
The stability result is coherent with what one could expect from the variational characterization. Take as an example the kink solution, $K= 4\arctan e^{x}$. For $q=0$, the kink is stable modulo translations. For $q<0$, the position of the kink minimizes the part of the energy with the $\delta$ term; hence the modified equation forces small perturbation to remain close to the energetically convenient centered kink. The argument goes in the other direction for $q>0$: translations are energetically convenient, so that the initial datum related to a translating kink is expected to leave the orbit of $(K,0)$.  
\end{remark}

\begin{remark}
Since the $\delta$ nonlinearity is not increasing, the focusing or defocusing behavior of the singular term does not depend only on the sign of $q$. If $\|u\|$ is small, the nonlinear $\delta$ term in \eqref{CauchyDeltaV} is approximated by the linear potential $q \delta_0(x) u$, which enforces stability for $q>0$. Theorem \ref{TeoSt} states the opposite nature for $q$ positive, when the linearization is done around the kink.
\end{remark}

\subsection*{Plan of the paper}

The paper is laid out as follows. Section \ref{secwp} is devoted to prove the global well-posedness of equation \eqref{sGdelta} in appropriate energy spaces. In Section \ref{secwaves} we construct stationary waves depending on different parameter values of $q \in \R$. The final Section \ref{secstab} contains the proof of the (in)stability result (Theorem \ref{TeoSt}).

\subsection*{On notation} Throughout the paper, for any Hilbert space $X$ we indicate by $\langle \cdot, \cdot \rangle_X$ the scalar product in $X$ and by $(\cdot, \cdot)_{X \times X^*}$ the natural coupling with the dual space $X^*$. The support of any set $G$ is denoted as $\mathrm{spt}(G)$. According to custom we denote $\N_0 = \N \cup \{0\}$.

\section{Global well posedeness}
\label{secwp}

This section is devoted to prove Theorem \ref{TeoExi}. The hard part of the proof comes from the low regularity of the $\delta$-potential. In a low regularity space, namely $C^0((-T,T) \times \R)$, local existence can be easily proved by a fixed-point argument. The $L^\infty_{t,x}$ norm is well suited to control the term with the irregular potential in the Duhamel's formulation.  On the other hand, for low regularity solutions there is no natural energy conservation, and hence global existence fails. We look first to a regularized problem, where the smooth potential guarantees global existence and a bound over the energy norm.  The delicate part is then to prove convergence of the solution of the approximate problem to a solution of  \eqref{CauchyDeltaV}.

\subsection{Global existence}

By time reversability of the equation, we restrict the argument of the proof to $t>0$. Recall that the linear Klein Gordon equation,
\begin{equation}     \label{CauchyLinV}\begin{cases}\partial_t u_1= u_2,
     \\\partial_t u_2= \partial_{xx} u_1 - u_1, 
     \end{cases}
\end{equation}
has a fundamental forward solution
\[
G(x,t)= \tfrac{1}{2} \theta(t-|x|) J_0(\sqrt{t^2- |x|^2}).
\]
Here $J_0$ is the Bessel function of the first kind and $\theta$ is the Heaviside cut-off function; see Appendix C in Komech and Komech \cite{KoKo07a} and the references therein. For any initial datum $\left(u_1^0,u_2^0\right) \in H^1\times L^2$, the unique solution for positive times of the Cauchy Problem \eqref{CauchyLinV} is given by the convolution 
\begin{equation}
\label{LinV}
     W(t)[u_1^0,u_2^0]=  \begin{pmatrix}
        u_1(x,t)\\ u_2(x,t) 
    \end{pmatrix} =\int_\R \begin{pmatrix}
        \dot G(x-y,t) & G(x-y,t) \\ \ddot G(x-y,t )&\dot G(x-y,t)  
    \end{pmatrix} \begin{pmatrix}
        u_1^0(y)\\u_2^0(y) 
    \end{pmatrix}\,dy.
\end{equation}
A solution to \eqref{CauchyDeltaV} is a fixed point of the integral operator 
\begin{align}
    \label{DuhamelCV}
    U[u_1, u_2](x,t):&= W(t)[v_1^0,v_2^0] + \int_0^t W(t-s)[0, F(v_1)] \,ds, \\ 
    F(v_1):&= \sin(v_1) -v_1 +q\delta_0 \sin(v_1).
\notag
\end{align}

For smooth mollifiers $\rho_\varepsilon$ that weakly-$*$ approximate the Dirac-$\delta$,  $\rho_\ep \stackrel{\ast}{\rightharpoonup}\delta$, the associated approximated problem is
\begin{equation}     
\label{CauchyAppV}
\begin{cases}\partial_t u_1= u_2,
     \\\partial_t u_2= \partial_{xx} u_1 - u_1 -(\sin (u_1) -u_1) - q\rho_\varepsilon(x) (\sin(u_1)\cdot \rho_\varepsilon). 
     \end{cases}
\end{equation}

\begin{proposition}
\label{PropExi}
For any initial datum $\left (u_1^0, u_2^0\right )  \in H^1_{\sin} \times L^2$ and every $\varepsilon > 0$ there exists a unique global solution $\left (u_1^\varepsilon, u_2^\varepsilon\right ) \in C((0,\infty); H^1_{\sin} \times L^2)$ to the Cauchy problem \eqref{CauchyAppV}, which verifies
\begin{equation}
    \label{Energyineq}
     \sup_{t > 0} \left (\|\sin u_1^\varepsilon(t)\|_{L^2_x}+\| u_2^\varepsilon(t)\|_{ L^2_x}+\|\partial_x u_1^\varepsilon(t)\|_{L^2_x}\right )\le C \| (u_0, u_1) \|_{H^1_{\sin} \times L^2},
\end{equation}
where $C > 0$ is a uniform constant independent of $\varepsilon$.
\end{proposition}
\begin{proof} We divide the proof into several steps.\smallskip

\noindent 1. Integrability theory. As in the classical sine-Gordon theory, we reduce the Cauchy problem \eqref{CauchyAppV} in $H^1 \times L^2$. For any $u_0 \in H^1_{\sin}$ there exists $f\in H^2_{\sin}$ such that 
\begin{equation*}
    u_1^0 - f \in H^1, \qquad \text{and,} \qquad \|f\|_{\dot{H}^2} + \|\sin f \|_{L^2} + \left \|u_1^0-f\right \|_{L^2}\le C \left \|u_1^0\right \|_{H^1_{\sin}}.
\end{equation*}
For $f$ fixed, the solution $(u_1, u_2)$ to $\eqref{CauchyAppV}$ corresponds to $v_1:= u_1-f, v_2:= u_2$, which is, in turn, the solution to the Cauchy Problem with initial datum in $H^1\times L^2$,
\begin{equation}
  \begin{cases}
    \label{CauchyL2V}
\partial_t v_1= v_2
     \\\partial_t v_2= \partial_{xx} v_1 + \partial f_{xx} - \sin (v_1+f)  - q \rho_\varepsilon \sin \left ( (v_1+f)\cdot \rho_\varepsilon\right ) \\
v_1^0:= u_1^0-f; \ \ \ v_2^0:=u_2
    \end{cases}
\end{equation}

\noindent 2. Local Existence. For a fixed initial datum $(v_1^0,v_2^0)$, the solution to \eqref{CauchyL2V} is a fixed point of the integral operator,
\begin{equation}
    \label{DuhamelL2V}
    U_\ep[u_1, u_2](x,t):= W(t)[v_1^0,v_2^0] + \int_0^t W(t-s)[0, F_\varepsilon(v_1)] \,ds,
\end{equation}
defined on the space
\begin{equation*}  
X_T= C^0\left ( [0,T); H^1\times L^2 \right),
\end{equation*}
for a certain $T>0$ to be specified later. The nonlinearity $F_\varepsilon$ is
\[F_\varepsilon(u_1)=\partial_{xx} f + \sin(u_1+f) - u_1(s) + q\rho_\varepsilon(x)\sin((u_1+f)\cdot \rho_\varepsilon)(s).
\]

Existence and uniqueness follow from $U_\ep $ being a contraction in $X_T$. For $(v_1,v_2), (w_1,w_2) \in X_T$, $0<t<T$ we have explicitly from \eqref{LinV}-\eqref{DuhamelL2V}
\begin{equation}
    U_\ep[v_1, v_2](t)- U_\ep[w_1, w_2](t) = \begin{pmatrix}
        I(x,t) \\ \partial_t I(x,t),
    \end{pmatrix}
\end{equation}
with $I$ given by
\[
I(x,t):= \int_0^t\int_\R G(x-y,t-s)\left( F_\varepsilon (v_1)(y,s) -F_\varepsilon(w_1)(y,s) \right )\,dy \,ds.
\]
Recall that 
\[
\|\rho_\varepsilon\|_{L^2} \le \varepsilon^{-\frac 1 2 }, \qquad \theta(t-|x|) \left| \partial_x^j \partial_t^j J_0(m\sqrt{t^2-x^2}) \right| \le C_{j,k} (1+t)^{j+k}, \quad \forall j,k\in \N_0.  
\]

Elementary trigonometric identities, Sobolev embedding, and Young and Jensen inequalities, lead to the estimate
\begin{align}
    &\|I(\cdot, t)\|_{L^2} \notag\\&\le C\left \| \int_0^t \! \int_\R J(x-y, t-s) \theta(t-s -|x-y| )\left(|w_1-v_1|(y,s)+ \|w_1-v_1\|_{L^\infty_{t,x}} \rho_\varepsilon(y)\right ) \,dyds \right \|_{L^2_x}\notag
    \\ &\le Ct^{\frac1 2 }\int_0^t (t-s)\|w_1- v_1\|_{L^2_x}+ \|w_1-v_1\|_{L^\infty_{t,x}}(t-s)^\frac 12 \,ds \\&
    \le C (t^\frac{5}2+ t^2)\|v_1-w_1\|_{L^\infty_tH^1_x}. \notag
\end{align}
For the derivative, notice that 
\[
\partial_x G(x-y,t-s)= \partial_x J(x-y, t-s) \theta(t-s-|x-y|) -\tfrac{1}{2} \sgn (x-y)\delta_{t-s-|x-y|}.
\]
Thus, we have
\begin{align*}
       \|\partial_xI(\cdot, t)\|_{L^2} &\le C\left \| \int_0^t \! \int_\R \partial_x  G(x-y,t-s)\left(|w_1-v_1|(y,s)+ \|w_1-v_1\|_{L^\infty_{t,x}} \rho_\varepsilon(y)\right ) \,dyds \right \|_{L^2_x} 
        \\& \notag\le C(t^\frac52 + t^2)\|v_1-w_1\|_{L^\infty_tH^1_x}+ \frac {t^\frac 32} {\varepsilon^\frac 12}  \|w_1-v_1\|_{L^\infty_{t,x}} + t^\frac 32 \|w_1-v_1\|_{L^\infty_t L^2_x}  \\&
\le \notag C(t^\frac52 + t^\frac32 +t^\frac32\varepsilon^{-\frac 12}) \|v_1-w_1\|_{L^\infty_tH^1_x}.
\end{align*}
The bound for $\partial_t I$ is similar. Hence, for $T=T\big(\varepsilon, \left \|(v_1^0, v_2^0)\right \|_{H^1\times L^2}\big)$ there exists a unique solution to \eqref{DuhamelL2V} in $X_T$.\\

\noindent 3. Global existence and energy conservation. For $(u_1,u_2)\in H^1\times L^2$, we define the approximated energy as
\begin{equation*}
    E_\varepsilon(u_1,u_2):= \int_\R \tfrac{1}{2} | u_2| + \tfrac{1}{2} |\partial_x u_1|^2 + (1- \cos u_1) \,dx + q \left ( 1-\cos (\rho_\varepsilon \cdot u_1) \right).
\end{equation*}

The Cauchy problem \eqref{CauchyAppV} preserves regularity and it is continuous with respect to the initial datum. Conservation of energy follows as, for any solution $(u_1^\varepsilon,u_2^\varepsilon)$
 to \eqref{CauchyAppV}, the equation formally implies
\begin{equation}
\label{ConsEne}
\frac{d}{dt} E_\varepsilon\left (u_1^\varepsilon(t), u_2^\varepsilon(t)\right ) = 0.
\end{equation}
From this we deduce the following bound \eqref{Energyineq} for $(u_1^\varepsilon(t),u_2^\varepsilon(t))$,
\[
   \int_\R \tfrac{1}{2} | u_2^\varepsilon(t)|^2 + \tfrac{1}{2} |\partial_x u_1^\varepsilon(t)|^2 + (1- \cos (u_1^\varepsilon)(t)) \,dx \le E_\varepsilon\left (u_1^0,u_2^0\right) + |q|,
\]
so that the solution can be extended globally in time. This finishes the proof.
\end{proof}

We pass now to the limiting argument for equation \eqref{CauchyDeltaV}.

\subsection{Proof of Theorem \ref{TeoExi}}
Once again, the proof is divided into steps.
\smallskip

\noindent 1. Existence. We look for a solution of \eqref{DuhamelCV}. By Proposition \ref{PropExi}, for any $\varepsilon>0 $ there exists a global solution $(u_1^\varepsilon,u_2^\varepsilon)$ to \eqref{CauchyAppV}; moreover up to subsequences we have weak convergence in $L^2\left (\R \times[0,T]  \right )$; indeed,
\begin{equation}
\label{weak}
    \sin u_1^\varepsilon  \overset{L^2 }{\rightharpoonup} \sin u_1, \qquad \partial_x u_1^\varepsilon \overset{L^2}{\rightharpoonup} \partial_x u_1, \qquad u_2^\varepsilon \overset{L^2}{\rightharpoonup}  u_2.
\end{equation}
Recall that $(u_1^\varepsilon,u_2^\varepsilon)$ satisfies, for any $x,t$,
\begin{equation}
\label{approx}
      \begin{pmatrix}u_1^\varepsilon\\ u_2^\varepsilon\end{pmatrix} (x,t) = W(t)[v_1^0,v_2^0] + \int_0^t W(t-s)[0, \tilde F_\varepsilon(v_1)] \,ds,
 \end{equation}
with $\tilde F_\ep(v):= \sin (v)- v + \rho_\ep (\sin (v) \cdot \rho_\ep$. For a fixed rectangle $R=[x_1,x_2]\times [0,T]$, $u_\varepsilon (x,t)$ converges to $u(x,t)$ a.e. in $R$. By weak convergence we also have, uniformly for $(x,t) \in R$, that
\begin{equation*}
   \lim_{\varepsilon\to 0} \int_0^t \int_\R G(x-y,t-s)(\sin u_1^\varepsilon-u_1^\varepsilon)(y,s)\,dy\, ds=  \int_0^t \int_\R G(x-y,t-s)(\sin u_1-u_1)(y,s)\,dy \, ds.
\end{equation*}

Define $\alpha(t)$ as the weak limit in $L^2_{[0,T]}$ of $ \sin((u_1^\varepsilon(t)) \cdot \rho_\varepsilon)$. 
We can pass to the limit for the $\delta$ approximation, yielding
\begin{align}
\label{ContoDelta}
    \sup_{\substack{0\le t\le T  \\x\in [x_1,x_2]}}& \left | \int_0^t \int_\R G(x-y,t-s) \left ( \sin (u_1^\varepsilon\cdot\rho_\varepsilon(s) ) \rho_\varepsilon(y) - \alpha(s) \delta_0(y)\right )\,dyds \right | \le  \\
    \notag  \sup_{\substack{0\le t\le T  \\x\in [x_1,x_2]}} &\left |\int_0^t \left(\sin \left(u_1^\varepsilon\cdot\rho_\varepsilon(s) \right )- \alpha(s)\right )\int_\R G(x-y,t-s) \rho_\varepsilon(y)\,dyds\right | +\\& + \sup_{\substack{0\le t\le T  \\x\in [x_1,x_2]}}\left|\int_0^t \alpha(s) \int_\R G(x-y,t-s) \left (\rho_\varepsilon(y)- \delta_0(y) \right )\,dyds\right | \nonumber
\end{align}
The two integrals converge to zero when $\varepsilon \to 0$ uniformly for $(x,t) \in R$: the first by definition of $\rho_\varepsilon$, the second one because weak convergence is uniform over sets compact in the strong topology. Passing to $\varepsilon \to 0 $ in \eqref{DuhamelL2V} for a.e. $(x,t)\in R$ we get
\begin{equation}
     \label{LimApp}
      \begin{pmatrix}u_1\\ u_2\end{pmatrix} (x,t) = W(t)[v_1^0,v_2^0] + \int_0^t W(t-s)[0, \sin u_1- u_1 + q\alpha \delta_0] \,ds.
\end{equation}
The limit is uniform in $(x,t)\in R$; hence $\alpha(s) = \sin (u_1(0,s))$ and $(u_1,u_2)$ satisfies \eqref{DuhamelCV}.\\

\noindent 2. Uniqueness and energy conservation. For a certain $T=T\left( u_1^0,u_2^0\right)>0$, there exists a unique solution $(u_1,u_2)\in X_T$ for the Cauchy problem \eqref{CauchyDeltaV}, 
expressed by Duhamel formula \eqref{DuhamelCV}. As in Proposition \ref{PropExi}, we prove local existence and uniqueness by contraction of the operator
\begin{equation}
    \label{DuhamelDV}
    U[u_1, u_2](x,t):= W(t)[v_1^0,v_2^0] + \int_0^t W(t-s)[0, F(u_1)] \,ds.
\end{equation}
The computation follows the same steps as in Proposition \ref{PropExi}, with minor modifications that repeat the proof by Komech and Komech \cite{KoKo07a}, so we omit the details. 
As we have already proved existence in $X_T$ for any $T$, the solutions must coincide and verify relation \eqref{Energyineq2} by the weak convergence in \eqref{Energyineq}. Conservation for $(u_1^\varepsilon, u_2^\varepsilon) $ in Proposition \ref{PropExi}, the weak convergences \eqref{weak} 
and uniform convergence on compact sets imply for $t\ge 0$ that
\begin{equation*}
    E(u_1,u_2)(t) \le \lim_{\varepsilon \to 0 } E_\varepsilon (u_\varepsilon)(t) = \lim_{\varepsilon \to 0} E_\varepsilon^0\left (u_1^0, u_2^0\right ) = E\left(u_1^0,u_2^0\right).
\end{equation*}

The Cauchy problem \eqref{CauchyDeltaV} can be solved backward in time, and by uniqueness the solution with initial datum $(u_1(t), u_2(t))$ at time $T=-t$ coincides with $(u_1^0, u_2^0)$. 
Repetition of the above computation leads to the reversed inequality; the two inequalities combined imply conservation of energy. By \eqref{Energyineq2}, the solution can be extended globally. This finishes the proof.
\qed

\begin{remark}
Albeit in the last step of the proof of Theorem \ref{TeoExi} we have directly exhibited a unique solution for the Cauchy problem \eqref{CauchyDeltaV}, the construction of the approximated problem \eqref{CauchyAppV} is still necessary in order to prove regularity, to obtain the norm bound \eqref{Energyineq2} for $(u_1,u_2)$ and, consequently, to show global existence. Conversely, Proposition \ref{PropExi} states the uniqueness of the solution to problem \eqref{CauchyAppV}; deducing from this fact the uniqueness for the limit \eqref{CauchyDeltaV} would require stronger compactness properties, which are significantly harder to obtain than the simple argument presented in Theorem \ref{TeoExi}.
\end{remark}
%
%

\section{Stationary waves}
\label{secwaves}

In this section we prove Propositions \ref{CarSW} and \ref{MinSW}. First we state two Lemmata, which constitute the base of the constructive argument for Proposition \ref{CarSW}.  The first one decomposes a solution of \eqref{staticDelta} as a solution of 
\begin{equation} 
- u_{xx} + \sin u = 0, \label{SGTI} 
\end{equation}
satisfying a nonlinear gluing condition in the origin, while the second recalls the expression of integrable solution 
of \eqref{SGTI}.

\begin{definition}
\label{Defso}
    We say that $u\in H^1_{\sin}(\R)$ satisfies equation \eqref{staticDelta} if for any test function $\varphi \in H^1$ there holds
    \begin{equation}
        \label{weak del}
        \int_\R u_x \varphi_x \,dx+ \int_\R \sin u \varphi \,dx + q \sin (u(0)) \varphi(0)=0.
    \end{equation}
We say $u$ solves \eqref{staticDelta} in $\R\setminus \left\{ 0 \right \}$ if the above equality holds for any $\varphi \in H^1$ with $0\notin \mathrm{spt}\ \varphi$. 
\end{definition}

\begin{lemma}
\label{Lemma1}
    A function $u\in H^1_{\sin}(\R)$ is a weak solution of equation \eqref{staticDelta} if and only if the following statements hold:
\begin{enumerate}
    \item $u$ solves \eqref{staticDelta} in $\R\setminus\left \{0\right \}$.
    \item We have the limit, 
    \[
    q\sin (u(0)) = \lim_{\varepsilon \to 0}u_x(\varepsilon)-u_x(-\varepsilon) := u_x(0^+)- u_x(0^-).
    \]
\end{enumerate}
\begin{remark}
 A solution $u$ of \eqref{staticDelta} in $\R\setminus \left\{ 0 \right \}$ as in Definition \ref{Defso} solves \eqref{SGTI} in the same set. Standard regularity arguments imply $u\in C^\infty (\R \setminus (-\varepsilon, \varepsilon))$ for any $\varepsilon>0$, and the evaluation $u_x(\pm \varepsilon)$ makes sense.
\end{remark}
\begin{proof} 
A solution of \eqref{staticDelta} solves the same equation in $\R\setminus\left \{0\right \}$. This shows property (1). Now, for $\varphi$ smooth with compact support, we test its rescalation, $\varphi^r(x):= \varphi (rx)$, against the solution $u$ to get 
\begin{equation*}
    \int_\R u_x \varphi^r_x \,dx+ \int_\R \sin u \varphi^r \,dx + q \sin (u(0)) \varphi(0)=0.
\end{equation*}

By dominated convergence the second integral goes to zero when $r \to \infty$. To estimate the first one we notice that, for any $a_r, b_r >0$,
\begin{align}
    \int_\R r \varphi_x(rx) u_x(x)& = \int_{\R\setminus [-a_r, b_r]} \varphi_x (y) u_x\left ( \frac{y}{r}\right )  + \int_{-a_r}^{b_r}     \label{intpp}
\varphi_x (y) u_x\left ( \frac{y}{r}\right )\\& = \notag-\int_{\R\setminus [-a_r, b_r]} \frac{\varphi (y) u_{xx}\left (\frac y r\right )}{r} +\left [\varphi u_x\left (\frac y r\right ) \right ]^{-a_r}_{b_r}  +\int_{-a_r }^{b_r}\varphi_x (y) u_x\left ( \frac{y}{r}\right ).
\end{align}
Regularity of $u$ in the set not containing $x=0$ allows to integrate by parts. The two integral quantities are estimated, by H\"older inequality, as
\begin{align*}
  &\left |\int_{\R\setminus [-a_r, b_r]}\frac{\varphi (y) u_{xx} \left (\frac y r\right )}{r} \right | \le \frac{\|\varphi\|_{L^2(\R)}}{r^{\frac 1 2 }}\|u_{xx}\|_{L^2(\R \setminus (a_r, b_r)) } \le  \frac{\|\varphi\|_{L^2(\R) }}{r^{\frac 1 2 }}\|\sin u\|_{L^2(\R) },\\&
\left |\int_{-a_r }^{b_r}\varphi_x (y) u_x\left ( \frac{y}{r}\right )\right |  \le \|\varphi_x\|_{L^2(-a_r,b_r)} \|u_x\|_{L^2(\R) } r^{\frac 1 2},
\end{align*} 
which go to zero when $a_r, b_r \to 0$ fast enough for $r\to \infty$. By generality of $\varphi (0)$ we have proved property (2). 

For the converse implication, consider a smooth $\chi$ with spt$\chi = [-1,1]$, and $\chi \equiv 1$ on $[-1/2, 1/2]$, and its rescalation $\chi^r(x):= \chi(rx)$. For a generic test function $\varphi \in H^1$ and for any $r>0$, by property (1) we have
\begin{equation}
    \int_\R u_x \varphi_x \,dx+ \int_\R \sin u \varphi  \,dx= \int_\R u_x (\chi^r\varphi)_x \,dx+ \int_\R \sin u (\chi^r\varphi ) \,dx. 
\end{equation}
The second integral on the right hand side goes to zero for $r\to \infty$. For the first one we can repeat the same calculations as in \eqref{intpp} for $r\to \infty$; the previous equality then reads
\begin{equation*}
    \int_\R u_x \varphi_x \,dx+ \int_\R \sin u \varphi  \,dx=\varphi(0)( u_x(0^-)- u_x(0^+)),
\end{equation*}
and, by property (2), this implies that $u$ is a solution of \eqref{staticDelta}.
\end{proof}
\end{lemma}

\begin{lemma}
\label{Lemma2}
    Let $u\in H^1_{\sin}$ be a non trivial solution of \eqref{sG} on $(-\infty, -c)$ (resp.  $(c, \infty)$) for some $c>0$. Then either $u(x)=K_{x_0} +2k\pi$ or $u(x)=K_{x_0}(-x)+ 2k\pi$ for some $x_0\in \R, k \in \mathbb{Z}$ on the whole interval $(-\infty, -c)$ (resp. $(c, \infty)$). 
\end{lemma}
\begin{proof}
We adapt to our context a calculation by Drazin and Johnson \cite{DrJ89}. Without loss of generality we restrict to $(-\infty, -c)$. Since $u$ is a smooth non zero solution, we must have a certain $(a,b) \subset (-\infty, c)$ such that $u_x \neq 0 $ in $(a,b)$. 
Multiplying \eqref{sG} by $u_x$ and integrating leads to 
\begin{equation}
    \tfrac{1}{2} u_x^2(x)= -\cos u(x) -C= 2\sin^2\left ( \tfrac{1}{2} u(x) \right ) -C, \qquad  \text{for } x\in (a,b).
    \label{solsg}
\end{equation}
If the constant $C$ is non zero, the sum $|u_x|^2 +| \sin u |^2 $ remains greater than a positive constant on $(a,b)$. 
By integrability condition, the interval must have finite measure, with $u_x(a)=0$. 

On the other hand, by smoothness of $u$, the solution satisfies equation \eqref{solsg} with the same constant $C$ on some interval $(a', a)$ where the first derivative is positive. Repeating the argument, we would have that $u$ is a periodic wave, contradicting $u \in H^1_{\sin}$. If $C=0$ then a standard integration argument leads to the desired result. 
\end{proof}

\begin{proof}[Proof of Proposition \ref{CarSW}]
We first prove the statement for $Q\in H^1$. In the region $x\le -\varepsilon <0$, by Lemma \ref{Lemma1} a solution $Q$ of the perturbed equation \eqref{sGdelta} solves the classical sine-Gordon equation \eqref{sG}. 
If $Q \in H^1$ is positive, by Lemma \ref{Lemma2} $Q(x)= K_{x_0}(x)$ for all $x<0$, where $x_0$ is a free translation parameter. 
Correspondingly for $x\ge \varepsilon>0$ we have $Q(x)=K_{x_1}(-x)$, with $x_1 \in \R$. By continuity in $x=0$ of $Q\in H^1$ it holds $x_0=x_1 $. Again by Lemma \ref{Lemma1}, it only remains to prove that 
\begin{equation}
\label{compat}
    Q'(0^-) - Q'(0^+) = - q \sin (Q(0)).
\end{equation}
By the explicit formula for $Q$ the equation reads
\begin{equation}
\label{compcon}
  -  8\frac{e^{-x_0}}{1+e^{-2x_0}} = q \sin (4\arctan e^{-x_0})
\end{equation}
Define $y:=e^{-x_0}>0$; by elementary trigonometric identities the above equality reads 
\begin{equation}
    \label{y2} -\frac{2}{q}= \frac{1-y^2}{1+y^2},
\end{equation}
and we can solve \eqref{compcon} as an equation in $y$, with a parameter $q$. The equation has one solution if $|q|>2$, given by 
\begin{equation*}
    y=\sqrt{\frac{q+2}{q-2}},
\end{equation*}
and no solution for $|q|\le2$.

We now prove the result for the kink. By Lemma \ref{Lemma2} a solution $u\in H^1_{\sin}$ which satisfies \eqref{sG} must verify $u=K_{x_0}$ on $(-\infty, -c)$ and $u= K_{x_0}$ on $(c,\infty)$ 
for any $c>0$, the translation parameters being the same by continuity. Since $K$ is a smooth function, $K_{x_0}'(0^+) = K_{x_0}'(0^-)$; the condition (2) in Lemma \ref{Lemma1} simplifies to 
\begin{equation*}
    0= q\sin (K_{x_0}(0)),
\end{equation*}
which for any $q \neq 0$ is satisfied only by $x_0=0$.
\end{proof}

For $q<0$, we show  a variational characterization of the stationary waves. The sign of $q$ forces an optimal competition between different terms of the energy for the construction of a nontrivial minimizer.

\begin{proof}[Proof of Proposition \ref{MinSW}]
The only negative term of $E$ is $q(1-\cos(v(0))$ so the energy is bounded below by $q$. For the same reason it is also coercive: if $E(v) \le C$, then $\|v\|_{H^1}^2\le 2C +| q|$.

As before, we consider the ground state case first. Up to subsequences, any minimizing $v_n$ converges weakly to $v$ in $H^1$; by the compact embedding in one dimension we also have $v_n \to v$ pointwise in any compact set. Hence, the energy is lower semicontinuous with respect to $H^1$ weak convergence, and $v$ is a minimizer for \eqref{MinGS}.

It is then enough to prove that for $-2<q<0$ the minimum of the energy is zero, while for $q<-2$ there exist functions $v\in H^1$ with $E(v)<0$; the case $q=-2$ is reached by a limit argument for $q_n \downarrow -2$.
 We recall that for $f\in H^1(\R)$ and $a>0$ we have the inequality
\begin{equation}
\label{Gagl}
    2\|f\|_{L^\infty}^2 \le a\|f\|_{L^2}^2 + \frac{1}{a}\|f_x\|_{L^2}^2,
\end{equation}
where the constant $2$ is sharp and attained. Consider at first $-2<q<0$, and suppose there exists a $v\in H^1$ with negative energy; basic trigonometric inequalities imply
\begin{equation*}
      \int_\R v_x^2 +  v^2  - \tfrac{1}{12} v^4 \,dx  \le \int_\R v_x^2 \, dx + 2 \int_\R1-\cos (v) \, dx< 2|q| \left (1-\cos(v(0))\right) \le |q|v(0)^2. 
\end{equation*}
Now for $v_b(x):=bv(x)$ with $b>0$ we have
\begin{equation*}
     \int_\R (v_b)_x^2 +v_b^2 \, dx - |q| v_b^2(0) < \frac{1}{12 b^2 } \int_\R v^4_b \,dx \le \frac{C}{b^2} \left ( \int_\R (v_b)_x^2 +v_b^2 \, dx \right ).
\end{equation*}
For $b$ sufficiently large and a certain $|q|<|q'| <2$ we would have a $v_b\in H^1$ satisfying
\begin{equation*}
    |q'|\|v_b\|_{L^\infty}^2 \ge \|v_b\|_{L^2}^2 + \|(v_b)_x\|_{L^2}^2
\end{equation*}
contradicting \eqref{Gagl}.

Now for $q<-2$, let $v$ be the extreme function satisfying \eqref{Gagl} as an equality, for $a=1$ and $v(0)= \|v\|_{L^\infty}$. Define $v_b(x):= \frac{v(x)}{b}$; by equality \eqref{Gagl} we can express the energy of $v_b$ as 
    \begin{equation*}
    E(v_b)=   v_b^2(0) +q (1-\cos v_b(0)) + \int_\R - \tfrac{1}{2} v_b^2 + (1-\cos v_b )\,dx.
\end{equation*}
Taylor expansion and $q<-2$ imply that, for $b$ large enough, $E(v_b)<0$. There exists then a nonzero $w\in H^1$ point of minimum for \eqref{MinGS}, which satisfies the associated Euler-Lagrange equation \eqref{staticDelta}. By Proposition \ref{CarSW}, $w= Q$.

For the second affirmation, define the space
\[
X:=\left \{ v\in H^1_{\sin} \, : \, u(-\infty ) = 0, \; u (\infty )= 2\pi \right \},
\] 
and notice that
 \begin{equation}
 \label{MinK2}
 \begin{aligned}
          \inf\limits_{v\in X}\int_\R \tfrac{1}{2} v_x^2 + (1-\cos v) \,dx &+  \inf\limits_{v\in X} q(1-\cos v(0)) \leq \\
          & \inf\limits_{v\in X}\int_R \tfrac{1}{2} v_x^2 + (1-\cos v ) \,dx + q(1-\cos(v(0)))
 \end{aligned}         
    \end{equation}
    For $q<0$, the second infimum on the left hand side is reached by any $v\in X$ such that $v(0) = \pi$. Let $v_n  \in X$ be a minimizing sequence of the first infimum problem; we can assume by translation invariance of the problem that for any $n$
    \begin{equation}
    \label{C}
        \int_{-\infty} ^0 1-\cos v_n \,dx = \int_{0}^\infty 1-\cos v_n \,dx.
    \end{equation}
    
    Boundedness of the integral energy for $v_n$ implies weak convergence in $H^1_{\sin} $ to $v$; moreover the sequence converges uniformly as a continuous function on any compact set $ \tilde{K}\subset \R$ by Sobolev embedding. 
Together with condition \eqref{C} this leads to $v \in X$. Lower semicontinuity of the integral energy assures that $v$ is the minimizer, and that it solves the associated Euler-Lagrange equation \eqref{SGTI}. 
Thus $v$ is a sine-Gordon kink. By direct inspection, the kink $K_{x_0}$ is a minimizer for both the problems on the left hand side of \eqref{MinK2}, and hence is a minimum for \eqref{MinKink}.
\end{proof}

\begin{remark}
The stability of the ground states $(Q,0)$ could follow as a corollary of Proposition \ref{MinSW} without the need of any spectral argument, by the classical strategy from Cazenave and Lions \cite{CaLi82}. A configuration starting with energy close to the minimal energy of the ground state maintains this property for all times by energy conservation; and since all the minimizing sequences converge to $(Q,0)$ the evolution must remain close to the minimizer. Still, we consider the stability result in the next section to be more complete. In the first place, the stability notion in \ref{defst} is stronger. More importantly, the spectral study is preparatory to provide asymptotic stability results.
\end{remark}

\begin{remark}
In the case when $q<0$, there is another detail marking the difference between $-2<q<0$ and $q<-2$. If we linearize equation \eqref{sGdelta} near $u=0$,  we get the an evolution linear problem with the Schrödinger operator:
    \begin{equation*}
        (-\partial_{xx}+ 1 +q\delta_0) u =0.
    \end{equation*}
    For $q<0$, the operator has only one eigenfunction in the point spectrum, with eigenvalue 
    \begin{equation*}
        \lambda= 1- \left( \frac{q}{2} \right )^2.
    \end{equation*}
    
In the case when $-2<q<0$, the linearized problem around zero has a strictly positive eigenvalue at the bottom of its spectrum, and we expect the zero solution to be orbitally stable; instead for $q<-2$ the eigenvalue is negative, and the zero solution is spectrally unstable.
\end{remark}

\section{Stability Study}
\label{secstab}

In this section we prove Theorem \ref{TeoSt}. Since the spectral study is for the most part the same for either ground states or kinks, for simplicity in the notation we will call through this section $(K,0)$ as the generic static solution given by Proposition \ref{CarSW}, specifying where needed the differences if $K \in H^1$ or not, and the sign of $q$.

Evolution of small perturbation of stationary waves suggests to consider the following Cauchy problem 
\begin{equation}
  \begin{cases}
    \label{CauchyPert}
\partial_t v_1=v_2,\\
\partial_t v_2 = \partial_{xx}v_1 -\left (\sin(K+v_1) - \sin K\right )(1+q\delta_0(x)),
\\v_1|_{t=0}= v_1^0, \qquad v_2|_{t=0} = v_2^0.
\end{cases}
\end{equation}

For a perturbation which remains small over time, the dynamics is effectively described by the linearized evolution
\begin{equation}
  \begin{cases}
  \partial_t v_1=v_2,\\
  \partial_t v_2 = \partial_{xx}v_1 - (\cos K) v_1 -q (\cos K) v_1 \delta_0,
\label{LinPert}\\
v_1|_{t=0} = v_1^0,\qquad v|_{t=0}=v_2.
    \end{cases}
\end{equation}

This linear evolution system is governed by the linearized operator around the stationary wave, which formally reads as
\begin{equation}
    \label{lin} \tilde \cL_K:= - \partial_{xx} + \cos K + q (\cos K) \delta_0 : H^1 \to H^{-1}
\end{equation}

The spectral study of the last operator is complicated by the presence of the singular term $\delta_0$, which prevents from defining the operator in $H^2 \subset L^2 \to L^2$. 
To prove spectral properties we look at $\cL_K$, a proper restriction of $\tilde \cL_K$ with  $D(\cL_K) \subset H^2$ and image in $L^2$.

\begin{lemma}
    The operator $(\mathcal{L}_K, D(\cL_K))$ defined as 
      \begin{align}
           \label{domain}
\cL_K :&= -\partial_{xx} +\cos K, \\
   D(\mathcal{L}_K):&= \left \{ u \in H^2(\R\backslash \{0\}) \cap H^1 \, : \,\ u_x(0^+)- u_x(0^-) = q \cos(K(0)) u(0)\right \},
\notag
\end{align} is self-adjoint with respect to the standard $L^2(\R)$ product. Moreover, it coincides with $\tilde\cL_K$, i.e.
\begin{equation}
\label{eqoper}
\big( \tilde \cL_K u, v \big)_{H^{-1}; H^1}= \langle \mathcal L _K u, v \rangle_{L^2}, \qquad \forall\ u \in D(\cL_K), \; v\in H^1.
\end{equation}
\label{lemadj}
\end{lemma}
\begin{proof}
    Consider the following densely defined operator in $L^2$,
    \begin{equation} \label{defcalM}\begin{aligned}
    \mathcal{M} &:= -\partial_{xx} + \cos K, \\
    D(\mathcal{M}) &:= \left \{ u \in H^2(\R\setminus \left\{0 \right \}) \cap H^1 \, : \, u_x(0^+)- u_x(0^-) = 0 = u(0)\right \}.
    \end{aligned}
    \end{equation}
$\mathcal{M}$ is bounded from below, densely defined and closed by direct verification. Now for any $u,v\in H^2(\R\setminus\left \{ 0\right\})$ we have
\begin{equation}
    \left \langle \mathcal{M} u,v\right \rangle_{L^2}= -v_x(0^-)u(0^-) + v(0^-)u_x(0^-) +v_x(0^+)u(0^+) - v(0^+)u_x(0^+) + \left \langle u,\mathcal{M} v\right \rangle_{L^2},\label{adjoin}
\end{equation} which implies that $\mathcal{M}$ is symmetric over its domain. The formal adjoint operator is given by
 \begin{equation}
 \label{adj}
 \begin{aligned}
     \mathcal{M^*} &= -\partial_{xx} + \cos K, \\
    D(\mathcal{M^*}) &= H^2(\R\setminus \left\{0 \right \}) \cap H^1 (\R).
    \end{aligned}
\end{equation}
    
    In fact for any $u\in D(\mathcal{M})$ and $v\in D(\mathcal{M^*})$ we have, by \eqref{adjoin}, that
    \begin{equation*}
          \left \langle \mathcal{M} u,v\right \rangle_{L^2}=  \left \langle  u,\mathcal{M}v\right \rangle_{L^2}.
    \end{equation*}
To complete the domain description, suppose there exist $v, z\in L^2$ such that for any $u\in D(\cM)$ there holds  $\langle \cM u, v\rangle_{L^2} = \langle u,z\rangle_{L^2}$. Let $v_n\in D(\cM ^*)$ converge to $v$ in $L^2$. Then
\[ 
\langle u,z\rangle_{L^2} = \lim_n \langle \cM u, v_n\rangle_{L^2} = \lim_n \langle u,\cM v_n\rangle_{L^2}.
\]
The above convergence for all $u$ in a dense domain implies $\cM v_n \to z$ in $L^2$. It follows $v\in D(\cM^*)$ since $(\cM, D(\cM^*)$ is closed.
 
Hence (see Albeverio \emph{et al.} \cite{AGHH2ed}) all the self-adjoints extensions of $\mathcal{M}$ are given by the one parameter family 
      \[
      \begin{aligned}
      \mathcal L_Z &= -\partial_{xx} + \cos K, \\ 
    D(\mathcal L_Z) &= \left \{ u \in H^2(\R \backslash \{0\}) \cap H^1 \, : \, u_x(0^+)- u_x(0^-) = Z u(0)\right \},
    \end{aligned}
    \]
for $-\infty < Z\le \infty$. Taking $Z=q\cos K(0)$ we have the equality \eqref{eqoper} by direct computation.
\end{proof}

It is to be observed that system \eqref{LinPert} has the following vectorial representation,
\begin{equation}
    \label{PertLinMat}
    \begin{pmatrix}
        v_1\\ v_2 
    \end{pmatrix}_t =  JE_K \begin{pmatrix}
        v_1 \\ v_2
    \end{pmatrix},
\end{equation}
for $J, E_K$ matrices/operators defined as
\[
J:= \begin{pmatrix} 0 & 1 \\ -1 & 0 \end{pmatrix}, \qquad E_K := \begin{pmatrix} \mathcal{L}_K & 0 \\ 0 & 1 \end{pmatrix}, \quad JE_K : D(\cL_K) \times L^2 \to L^2\times L^2.
\] 

Let us recall the definition of spectral stability as well as a criterion for spectral instability.

\begin{definition}
Let $(K,0)$ be a stationary vector solution of \eqref{CauchyDeltaV} (i.e. $K$ is a solution of \eqref{staticDelta}).
We say that $(K,0)$ is \emph{spectrally stable} if the spectrum of the linearized operator
$JE_K$ satisfies $\sigma(JE_K) \subset i\mathbb{R} $.
Otherwise, $(K,0)$ is said to be \emph{spectrally unstable}.
\end{definition}

We recall the following spectral instability criterion from \cite{AnPl21a} (Theorem 3.2).

\begin{theorem}[instability criterion]
\label{InstabilityCrit}
Suppose the following assumptions hold:
\begin{enumerate}
   \item   $(\cL_K, D(\cL_K))$ is densely defined and self-adjoint in $L^2$.
    \item  $\mathcal{L}_K : D(\mathcal{L}_K) \to L^2(\R)$ is invertible with Morse index $n(\mathcal{L}_K) = 1$, and its spectrum satisfies
    \begin{equation} \label{SpecCond}
        \sigma(\mathcal L_K) = \{\lambda_0\} \cup J_0, \quad J_0 \subset [r_0, \infty), \quad r_0 > 0, \ \lambda_0 < 0.
    \end{equation}
    \item $JE_K$ is the generator of a $C_0$ semigroup $\{S(t)\}_{t \ge 0}$.
\end{enumerate}
Then the operator $JE_K$ has a real positive and a real negative eigenvalue.
\end{theorem}

\begin{lemma}
\label{lemma1n}
   Let $(K,0)$ be a stationary wave. Then $\mathcal{L}_K$ has at most one negative eigenvalue.
\end{lemma}
\begin{proof}
By Lemma \ref{lemadj}, $\mathcal{L}_K$ is a self-adjoint extension of the symmetric operator $\mathcal{M}$ with deficiency indices $n_\pm (\cM)=1$. We can prove  $\mathcal{M} \ge 0$. Recall the well known property for $\tilde K$ a sine-Gordon kink
\begin{equation}
\label{tildeK} -\partial_{xx} \tilde K_{x} + \cos(\tilde K) \tilde K_x= 0; \ \ \ \tilde K_x >0 
\end{equation}
By a standard Hardy trick (see e.g. \cite{IgNSZ15}, Appendix A) this implies the operator is positive, and $\tilde{K_x}$ is the unique $0$ eigenfunction. We adapt the calculation of \cite{IgNSZ15} to the present case: for $u\in D(\cM)$ with compact support, define $v:= |K_x|>0$ and letb $g:= u/v$. Then we have
\begin{align} \notag
\langle \cM u, u \rangle &= -\int_\R g^2 v v_{xx}+ 2gvg_xv_x + v^2 g g_{xx}  \, dx + \int_\R \cos K v^2 g^2 \, dx
\\ \notag &= -\int_\R g^2 v \left ( -\partial_{xx} + \cos K \right ) v \, dx + \int_\R g_x^2 v^2 \, dx + \left [ g v^2 g_x\right ]_{0^-}^{0+}
\\ \label{contoHardy}		& = \int_\R g_x v^2 \, dx + g(0)v^2(0) \left (g_x(0^-) - g_x(0^+)\right).
\end{align}

We have used identity \eqref{tildeK} for $v$, and standard integration by parts outside $0$. If $u\in D(\cM)$, $u(0)=0$ and this implies $g(0)=0$. Moreover
\begin{equation}
\label{BoundHardy} 0=u_x(0^+)- u_x(0^-) = v(0) \left ( g_x(0^+) - g_x(0^-) \right ).
\end{equation}
and hence the boundary terms in \eqref{contoHardy} dissapears. The thesis follows by Proposition A.3 in \cite{AnPl21a}.
\end{proof}

\begin{lemma}
\label{lemspec}
    Let $(K,0)$ be a stationary wave. It holds 
\[
\ker \cL_K = \left\{ 0\right \}, \quad\sigma_{\mathrm{ess}} (\cL_K)=[1,\infty).
\]
\end{lemma}
\begin{proof}
  All the self adjoint extensions $(\mathcal{L}_Z, D(\mathcal{L}_Z))$ for $-\infty <Z\le \infty$ have the same essential spectrum (see Proposition A.5 in \cite{AnPl21a}). For $Z=0$, the operator coincides with $\mathcal{L}_0=-\partial_{xx} + \cos K$ over the domain $D(\mathcal{L}_0)=H^2$. By Weyl's Theorem for compact perturbations 
\[\sigma_{\mathrm{ess}} (\mathcal{L}_0) = [1,\infty).\]

Suppose there exists a nonzero $v\in D(\mathcal{L}_K)$ with $\mathcal{L}_K v=0$. By \eqref{tildeK}, 
Sturm-Liouville theory for ODE's on half line implies $v=a^- K_x$ for $x<0$, $v=a^+ K_x$ for $x>0$, $a^-, a^+$ constants. If $K$ is a kink as in Proposition \ref{CarSW}, continuity in $0$ implies $a^+=a^-$; but then $v$ is smooth and it cannot verify the domain condition in \eqref{domain}. Similarly, if $K$ is a ground state, $a^+=-a^-$ by continuity and the domain condition in \eqref{domain} would imply the equality 
    \begin{equation*}
        0= 2K_{xx} (0^-) + q K_x(0) \cos(K(0)) = 2\sin K(0) + q K_x (0)\cos(K(0)).
    \end{equation*}
    But then by the expression of $K$ for $|q|>2,$ $y$ given by \eqref{y2},
    \begin{align}\notag
     2 \sin K(0) + q K_x(0) \cos(K(0)) &=  2  \sin (4 \arctan(y)) + 4q \cos(\arctan(y)) \frac y {1+y^2}\\&= - \frac{2(q-2)\sqrt{q^2-4}}{q^2} 
      \label{contoq}
    \end{align}
   and the domain condition is not verified.
\end{proof}

\begin{lemma}
\label{C0Lemma}
The operator $(\cA:= JE_K, D(\cA) := D(\cL_K )\times H^1)$ generates a $C_0$ semigroup $e^{\cA t}$ over $H^1\times L^2$. The Cauchy problem \eqref{PertLinMat} has a global solution, 
for initial value $(v_1^0, v_2^0)\in H^1 \times L^2 $ given by 
\[ \begin{pmatrix} v_1(t)\\ v_2(t) \end{pmatrix} = e^{\cA t}\begin{pmatrix} v_1^0\\ v_2^0 \end{pmatrix}.\]
If  $(v_1^0, v_2^0)\in D(\cA)$, then the solution belongs to   $ C( [0,\infty); D (\cA)) \cap C^1( [0,\infty); H^1\times L^2)$
Moreover it satisfies for some $C, \beta>0$ independent on $(v_1^0,v_2^0)$, and all $t>0$
\begin{equation}
\label{bounC0}
\|(v_1(t),v_2(t))\|_{H^1\times L^2} \le Ce^{\beta t} \|(v_1^0,v_2^0)\|_{H^1\times L^2}.
\end{equation}
\end{lemma}

\begin{proof}
Consider $\beta>0$ such that the biliear form 
\begin{equation*}
\langle u, v\rangle_{X_\beta} := \int_\R v'u' + \beta uv + q \cos K(0) v(0) u(0); \ \ \ u,v\in H^1
\end{equation*}
induces a norm equivalent to $\|\cdot \|_{H^1}$. Let the operator $\caB := \cA - (\beta+1) I$ be defined over $D(\cA)$. We prove explicitly $\caB$ is dissipative: for $\mathbf{v}= (v_1,v_2) \in D(\cA)$ we have
\[
\begin{aligned}
\langle -\caB \mathbf{v}, \mathbf{v} \rangle_{X_\beta \times L^2} =& \langle -v_2, v_1\rangle_{X_\beta} + \langle \mathcal{L}_K v_1, v_2\rangle_{L^2} + \beta \left ( \|v_1\|_{H^1}^2 + \|v_2\|_{L^2}^2\right) 
\\=& -\beta \langle v_2, v_1\rangle_{L^2} + \langle \cos K v_1, v_2 \rangle_{L^2}+(\beta+1) \left ( \|v_1\|_{H^1}^2 + \|v_2\|_{L^2}^2\right) 
\\=& (\beta+ 1) (\|v_1\|_{H^1}^2 + \|v_2\|_{L^2}^2) - \langle (-\beta + \cos K )v_2, v_1\rangle_{L^2} \ge 0
\end{aligned}
\]
Recall that a complex $\lambda$ belongs to $\sigma( \cA, D(\cA))$ if and only if $-\lambda^2 \in \sigma (\mathcal{L}_K, D ( \cL _K ) $. By specrtal study of $\cL_K$ and definition of $\caB$ we can pick $\mu>0$ 
such that $ \caB - \mu I$ is surjective. By Lumer-Phillips theorem $\caB$ is the generator of a $C_0$ semigroup of contractions $\left \{ W(t) \right \}$ 
over $H^1\times L^2$, satysfing $\|W(t)\|_{H^1\times L^2}\le 1.$ 
Finally, define the $C_0$ semigroup $e^{\cA t}:= e^{\beta+1}t W(t)$ generated by $\cA$. Verification of \eqref{bounC0} is direct.
\end{proof}

\begin{lemma}
Let $(K,0)$ be a stationary wave. Consider the bilinear form
\begin{equation*}
    Q_K(v, w):= \int v_x w_x + \cos K v w \,dx + q \cos(K(0)) v(0)w(0).
\end{equation*}
Then $Q$ is symmetric, closed and bounded below over $H^1 \times H^1$ and satisfies for any $(v_1, v_2) \in H^1\times L^2$
\begin{equation*}
    E(v_1+ K, v_2)= E(K,0) +\frac{1}2 Q_K(v_1,v_1) + \frac{1}{2} \left \| v_2\right \|_{L^2}^2+ o(\| v_1\|_{H^1}^2).
\end{equation*}
\end{lemma}
\begin{proof}
    The first statement is direct from the definition of $Q_K$. For the second consider the Taylor expansion 
    \begin{align}
        E(K+v_1,v_2)& = E(K, 0 ) + \int v_x K_x + \sin(K) v\,dx + q  v(0)\sin K(0) \nonumber\\ 
        &+\frac{Q_K(v_1,v_1)}{2}  + \frac{\left \| v_2\right \|_{L^2}^2}{2}  + o\left (\|v\|_{H^1}^2\right )  \nonumber
      \\& =  E(K,0)+ +\frac{1}2 Q_K(v_1,v_1) + \frac{1}{2} \left \| v_2\right \|_{L^2}^2+ o(\| v_1\|_{H^1}^2 ).
    \label{alli}
    \end{align}
In the last equality we have used \eqref{staticDelta} to cancel the first order term in $v$.
\end{proof}

\begin{lemma}
\label{Q=L}
Consider $(\mathcal{L_K}, D(\mathcal{L}_K))$ the self adjoint operator given in Lemma \ref{lemadj}. It coincides with $Q_K$, i.e. it holds
\begin{equation}
\label{calcu}
    (\mathcal{L}_K u, v)_{H^1}= Q(u,v) \ \ \ \forall \ u\in D(\mathcal{L}),  \ v \in H^1. 
\end{equation}
Moreover $\mathcal L_K$ has the same bound from below of $Q_K$: 
\begin{equation}
\inf_{0\neq u \in D(\cL)} \frac { \langle \cL u, u \rangle_{L^2} } {\|u\|_{L^2}^2} = \inf_{0\neq u \in H^1} \frac {Q_K(u,u)}{\|u\|_{L^2}^2}
\end{equation}
\end{lemma}
\begin{proof}
    The calculation for \eqref{calcu} is explicit; for the other statement see Kato \cite[Chapter VI, Theorems 2.1 and 2.6]{Kat80}.
\end{proof}

We prove existence (resp. nonexistence) of a negative eigenvalue for the operator $(\mathcal{L}_Z, D(\mathcal{L}_Z)$ for $q>0$ (resp. $q<0$).
\begin{lemma}
\label{Negative}
    Let $q>0$, $(K,0)$ a stationary wave and $Q_K$ the bilinear form associated to $K$. Then there exists $v\in H^1$ such that 
    \[Q_K(v,v) <0\]
\end{lemma}
\begin{proof}
Define $v:= |K_x|$, strictly positive and in $H^1$. We have
    \begin{align}
      Q_K(v,v) & =\int v_x^2 + \cos Kv^2\,dx + q v^2(0) \cos(K(0))\\  &= \int -v_{xx}v + \cos (K)v^2 \,dx - v(0^+) v_x(0^+)  + v(0^-)v_x(0^-) + q v^2(0) \cos(K(0))\notag
\\& = (v_x(0^-)-v_x(0^+))v(0) + qv^2(0) \cos(K(0))
    \end{align}
    The integral term in the last equality vanishes due to \eqref{tildeK}.
If $K$ is a kink, we are done as $\cos(K(0)) = -1<0$ and the other boundary term vanishes by smoothness.
    For $K$ a ground state, we have already evaluated the boundary term in \eqref{contoq}, which is negative for $q>2$.
\end{proof}

\begin{corollary}
\label{Corq}
Let $q>0$ and $(K,0)$ be a stationary wave. The Morse index of $(\cL_K, D(\cL_K))$ is $n(\cL_K) = 1$.
\end{corollary}
\begin{proof}
From definition it is clear that the closed symmetric operator $(\mathcal{M}, D(\mathcal{M}))$ from \eqref{defcalM} is non-negative with deficiency indices $n_\pm(\mathcal{M}) = 1$ by direct calculation. From extension theory it is known that $n(\mathcal{L}_Z) \leq 1$ for any self-adjoint extension $(\mathcal{L}_Z, D(\mathcal{L}_Z))$ (see Proposition A.3 in \cite{AnPl21a}). The result follows by Lemmas \ref{Q=L} and \ref{Negative}.
\end{proof}

\begin{lemma}
\label{PositiveSpect}
Let $q<0$ and $(K,0)$ be a stationary wave. Then there exists $\ep_q>0$ such that 
\[ \sigma (\cL_K) \subset [ \ep_q, \infty)\]
\end{lemma}
\begin{proof}
$(\cL_K, D(\cL_K))$ is a self adjoint operator, bounded below and $\sigma_{\mathrm{ess}} (\cL_K) = [1,\infty)$. If the thesis is false, there exists $u\in D(\cL)$ such that
\begin{equation}
\label{TBDP} \langle \cL_K u,u \rangle_{L^2} \le 0  
\end{equation}
Let $v:= |K_x|$, and $g:= \frac uv$. Repeating the calculation as in \eqref{contoHardy} we get
\[ \langle \cL_K u,u \rangle_{L^2} = \int_\R v^2 g_x^2 +  g(0)v^2(0) \left (g_x(0^-) - g_x(0^+)\right)\]
Adapting the boundary evaluation \eqref{BoundHardy} to the conditions satisfied by $u\in D(\cL)$ we get 
\[ g(0)v^2(0) \left( g_x(0^-) - g_x(0^+) \right ) = g(0)^2 v(0) \left (q\cos(K(0))v(0) +  K_{xx}(0^+) -K_{xx} (0^-)\right )  
\] 
The right hand side has already been evaluetad in lemma \ref{lemma1n}, and whether $K$ is a kink or a ground state it is positive for $q<0$.
\end{proof}

\begin{proposition}
\label{PropSt}
Let $q>0$ (resp $q>2$). Then the the stationary waves $(K,0)$ (resp. $(K,0)$ and $(Q,0)$) are spectrally unstable. 
Conversely, let $q<0$ (resp $q<-2$). Then the the stationary waves $(K,0)$ (resp. $(K,0)$ and $(Q,0)$) are spectrally stable. 
\end{proposition}
\begin{proof}
For $q>0$, by Corollary \ref{Corq} $(\cL_K, D(\cL_K))$ has exactly one negative eigenvalue; togheter with Lemma \ref{lemspec} this implies 
the condition \eqref{SpecCond} is verified. Moreover the operator is self-adjoint (Lemma \ref{lemadj}) and 
$JE_K$ generates a $C_0$ semigroup (Lemma \ref{C0Lemma}); the result follows from Theorem \ref{InstabilityCrit}.
For $q<0$, the result is direct from definition of $JE_K$ and spectra of $\cL_K$ given in Lemma \ref{PositiveSpect}.
\end{proof}

\subsection{Nonlinear stability analysis}

In this Subsection we directly apply Proposition \ref{PropSt} to prove Theorem \ref{TeoSt} pertaining to nonlinear stability. For the instability argument, on the other hand, it is customary in the literature to invoke a very general (abstract) result by Henry \emph{et al.} \cite{HPW82}, which essentially links nonlinear (orbital) instability from spectral instability. 

\begin{theorem}[Henry \emph{et al.} \cite{HPW82}]
\label{HPW2}
Let $Y$ be a Banach space and $U \subset Y$ an open subset such that $0 \in U$.
Assume there exists a map $M : U \to Y$ with $M(0)=0$ and a continuous linear operator
$L : Y \to Y$ with spectral radius $r(L) > 1$ such that for some $p>1$,
\begin{equation}
\|M(y) - Ly\|_Y = O\!\left(\|y\|_Y^p\right)
\quad \text{as } y \to 0 . \label{estimateHPW}
\end{equation}
Then $y=0$ is an unstable fixed point of $M$. More precisely, there exists
$\varepsilon_0 > 0$ such that for every ball $B_\eta(0)\subset Y$ and every
$N_0 \in \mathbb{N}$, there exist $n \ge N_0$ and $y \in B_\eta(0)$ such that
\[
\|M^n(y)\|_Y \ge \varepsilon_0 .
\]
\end{theorem}

\begin{remark}
In applications to the instability of nonlinear waves, usually we require the data-solution map to be at least of class $C^2$ (see, e.g., Corollary 3 in \cite{AAP24} or Corollary 3.1 in \cite{AngNat16}), in order to conclude that we smoothly move away from the orbit generated by the wave (via translations). In the present case, however, since there is no translation invariance due to the delta potential, we are able to apply the nonlinear abstract result directly, as we shall see.
\end{remark}

%

To prove our instability result, we need the following Lemma.

\begin{lemma}
\label{LemmaPerC}
For any $\left (v_1^0, v_2^0\right) \in H^1 \times L^2$ there exists a unique global solution to \eqref{CauchyPert}, $(v_1,v_2)\in C^0\left (\R; H^1\times L^2\right )$. Moreover for any $T>0$, there exists $\ep_0$ small such that 
for any initial datum satisfying
\[
\left \|\left (v^1_0,v_2^0 \right ) \right \|_{H^1\times L^2} \le \ep_0.
\]
and for any $0<t<T$ there holds
\begin{equation}
\left \|(v_1(t), v_2(t))\right\|_{H^1\times L^2} \le C(T, \ep_0) \left \|\left (v_1^0, v_2^0\right)\right\|_{H^1\times L^2} 
\label{bound}
\end{equation}
\end{lemma}
\begin{proof}
$(K,0)$ is a static solution of \eqref{CauchyDeltaV}, and by Theorem \ref{TeoExi} there exists a unique solution $(w_1,w_2)$ of \eqref{CauchyDeltaV} with initial datum $(K+v_1^0, v_2^0)$. 
Hence $(v_1, v_2):= (w_1-K, w_2)$ is the unique global solution of \eqref{CauchyPert}.
Inequality \eqref{bound} comes from the bound over the linear propagator \eqref{bounC0} and from Gronwall's nonlinear inequality for the integral representation. 
\end{proof}

\begin{proof}[Proof of Theorem \ref{TeoSt}]
For $q>0$, let $\left (v_1^0, v_2^0 \right ) \in H^1\times L^2$ be the eigenfunction of $JE_K$ with eigenvalue $\lambda $ of positive real part. For $T>0$, define  the linear operator $L_T: H^1\times L^2\to H^1\times L^2 $ which 
associates to $\left (u_1^0, u_2^0\right)$ the solution at time $T$ of the Cauchy Linear problem \eqref{LinPert} with initial datum $\left (u_1^0, u_2^0 \right )$. 
Then it holds $L_T \left(v_1^0, v_2^0\right) =e^{\lambda T} \left (v_1^0, v_2^0 \right )$. For $T$ finite large enough, $L_T$ has an eigenvalue of modulus larger than $1$.
Consider now the nonlinear operator $N_T: H^1\times L^2\to H^1\times L^2$ which associates to $\left ( u_1^0, u_2^0\right )$ the unique solution of \eqref{CauchyPert} at time $T$. 
Since $(K,0)$ is a stationary wave, $(0,0)$ is a fixed point of $N_T$. By Duhamel expression
    \begin{align}
     \notag   \left \| N_T \left ( u_1^0, u_2^0\right) - L_T \left (u_1^0, u_2^0\right )\right\|_{H^1\times L^2 }& \le \left \| \int_0^T e^{\cA(t-s)} \mathbf{f}(u_1)(s)\,ds  \right \| _{H^1 \times L^2} 
\\ &\label{estfin} \le C(T) \|\left(  u_1(t), u_2(t)\right )\|_{L^\infty_T (H^1\times L^2 )}^2
    \end{align} 
    for $(u_1(t),u_2(t))$ solution of \eqref{CauchyPert} with initial datum $(u_1^0, u_2^0)$ and
    \begin{equation*}
        \mathbf{f}(u_1) = \begin{pmatrix}
            0\\ (\sin (u_1+ K)-\sin(K)- \cos K u_1)(1+q\delta_0)
        \end{pmatrix}
    \end{equation*}
With $T$ fixed, let $\ep_0$ be given in Lemma \ref{LemmaPerC}. For $\left \|\left (u^1_0,u_2^0 \right ) \right \|_{H^1\times L^2} \le \ep_0 $ we have the desired estimate \eqref{estimateHPW} from \eqref{bound}, \eqref{estfin}, 
and hence $(K,0)$ is unstable by Theorem \ref{HPW2}.
%
%

For the stability result, consider now $q<0$. By energy conservation of \ref{TeoExi} and coercivity of $Q_K$ for $q>0$ we have, for constants $C$ independent on the initial datum
\begin{align*}
C\left \|v_1^0,v_2^0 \right \|_{H^1\times L^2}^2 &\ge \left |  E\left (v_1^0 + K, v_2^0 \right ) - E(K,0) \right |
\\&=\left | E \left( v_1(t) + K , v_2(t)\right ) - E(K,0) \right |
\\&=\left | \frac 12 \|v_2(t)\|_{L^2}^2 + Q_K(v_1(t)) + o \left  ( \|v_1(t)\|_{H^1}^2\right ) \right |
\\& \ge \frac12 \|v_2(t)\|_{L^2}^2 + \ep_q \|v_1(t)\|_{H^1}^2 + o \left  ( \|v_1(t)\|_{H^1}^2\right ). 
\end{align*}
\end{proof}

\section*{Acknowledgements}
S. Moroni thanks the Department of Mathematics and Mechanics at IIMAS, UNAM for their hospitality and support during an academic visit in the Fall of 2024 when this work was initiated.

\section*{Summary Statement}

\subsection*{Funding declarations}
S. Moroni has been partially supported by the Basque Government through the BERC 2022-2025 program and IKUR program, by the project PID2023-146764NB-I00 funded by MICIU/AEI/10.13039/501100011033 and FEDER/EU, and by the Spanish State Research Agency through BCAM Severo Ochoa CEX2021-001142. The work of R. G. Plaza was partially supported by Secretar\'{\i}a de Ciencia, Humanidades, Tecnolog\'{\i}a e Innovaci\'{o}n (SECIHTI -- Ministry of Science, Humanities, Technology and Innovation), Mexico, grant CF-2023-G-122. 

\subsection*{Conflict of interest} The authors declare no conflict of interest.

\subsection*{Ethics declaration} Not applicable.

\subsection*{Data Availability} No datasets were generated or analyzed during the current study.

\def\cprime{$'\!\!$} \def\cprimel{$'\!$}




%

\end{document}